\documentclass[12pt, reqno]{amsart}
\usepackage{amsmath, amsthm, amscd, amsfonts, latexsym, amssymb, graphicx, color}
\usepackage[bookmarksnumbered, colorlinks, plainpages]{hyperref}

\makeatletter \oddsidemargin.9375in \evensidemargin \oddsidemargin
\newtheorem{thm}{Theorem}[section]
\newtheorem{lem}[thm]{Lemma}
\newtheorem{prop}[thm]{Proposition}
\newtheorem{cor}[thm]{Corollary}
\newtheorem{prob}[thm]{Problem}
\theoremstyle{definition}
\newtheorem{defn}[thm]{Definition}

\theoremstyle{remark}

\numberwithin{equation}{section}

\begin{document}

\setcounter{page}{1}


\title[2-D skew constacyclic codes]{2-D skew constacyclic codes over $R[x,y;\rho,\theta]$}
\author[Mostafanasab]{H. Mostafanasab}
\thanks{{\scriptsize
\hskip -0.4 true cm MSC(2010): Primary: 00A69; Secondary: 12E20; 94B05
 Published in Journal of Algebra and Related Topics, Vol. 4, No 2, (2016), pp 49-63}}
\begin{abstract}
For a finite field $\mathbb{F}_q$, the bivariate skew polynomial ring $\mathbb{F}_q[x,y;\rho,\theta]$
has been used to study codes \cite{XH}. In this paper, we give some characterizations of the ring $R[x,y;\rho,\theta]$ where $R$ is a commutative ring. We investigate 2-D skew $(\lambda_1,\lambda_2)$-constacyclic codes in the ring $R[x,y;\rho,\theta]/\langle x^l-\lambda_1,y^s-\lambda_2\rangle_{\mathit{l}}.$ Also, the dual of 2-D skew $(\lambda_1,\lambda_2)$-constacyclic codes is investigated.
\end{abstract}

\maketitle
\section{Introduction}

Cyclic codes are amongst the most studied algebraic codes.
Their structure is well known over finite fields \cite{M}. Recently codes over rings have
generated a lot of interest after a breakthrough paper by Hammons et al. \cite{HK} showed
that some well known binary non-linear codes are actually images of some linear
codes over $\mathbb{Z}_4$ under the Gray map.
Constacyclic codes over finite fields form a remarkable class of linear codes, as they include the important family of cyclic codes. Constacyclic codes also have practical applications as they can be efficiently encoded using simple shift registers. They have rich algebraic structures for efficient error detection and correction, which explains their preferred role in engineering.
In general, due to their rich algebraic structure, constacyclic codes have been studied over various
finite chain rings (see \cite{AN,BU1,D,D2,D3,DL,QZZ,UB}).
Two-dimensional (2-D) cyclic codes are generalizations of usual cyclic codes which 
were introduced by Ikai et al. \cite{IKK} and Imai \cite{I}. G\"{u}neri and \"{O}zbudak 
studied the relations between quasi-cyclic codes and 2-D cyclic codes \cite{GO}. Decoding problem for 2-D cyclic codes was studied
by some authors \cite{SH,S4,S,S2}. Polynomial rings and their ideals are essential to the construction and understanding of cyclic
codes. For the first time in \cite{BGU} non-commutative skew polynomial rings have been used to
construct (a generalization of) cyclic codes.
Skew-cyclic codes were introduced by Boucher et al. \cite{BU,BU2}. They considered the skew-cyclic
codes as ideals or submodules over skew polynomial rings and studied dual skew-cyclic codes.
Skew constacyclic codes have been investigated by by Boucher et al. in \cite{BSU} and Jitman et al. in \cite{JLU}.
Xiuli and Hongyan \cite{XH} generalized the 2-D cyclic codes to 2-D skew-cyclic codes. They studied the
structures and properties of 2-D skew-cyclic codes. Also, they built relationships between 2-D
skew-cyclic codes and other known codes.

Throughout this paper, let $R$ be a commutative ring. For two given automorphisms $\rho$ and $\theta$ of $R$, we consider the set of formal bivariate 
polynomials $$R[x,y;\rho,\theta]=\Big\{\sum\limits_{j=0}^{t}\sum\limits_{i=0}^{k}a_{i,j}x^iy^j| a_{i,j}\in R ~\mbox{and}~ k,t\in\mathbb{N}_0\Big\}$$ which forms a ring under the usual addition of polynomials and where the multiplication 
is defined using the rule $$ax^iy^j\star bx^ry^s=a\rho^i\theta^j(b)x^{i+r}y^{j+s},$$
and extended to all elements of $R[x,y;\rho,\theta]$ by associativity and distributivity. The ring $R[x,y;\rho,\theta]$ is called a
{\it bivariate skew polynomial ring over} $R$ and an element in $R[x,y;\rho,\theta]$ is called a {\it bivariate skew polynomial}.
It is easy to see that $R[x,y;\rho,\theta]$ is a non-commutative ring unless $\rho$ and $\theta$ are indentity automorphisms on $R$.
For a bivariate skew polynomial $f(x,y)$ in $R[x,y;\rho,\theta]$, let $\langle f(x,y)\rangle_{\mathit{l}}$ denote the left ideal of $R[x,y;\rho,\theta]$
generated by $f(x,y)$. Note that $\langle f(x,y)\rangle_{\mathit{l}}$ does not need to be two-sided.

In section 2, we give some characterizations of the ring $R[x,y;\rho,\theta]$. In section 3, we introduce and investigate 2-D skew $(\lambda_1,\lambda_2)$-constacyclic codes in $R[x,y;\rho,\theta]/\langle x^l-\lambda_1,y^s-\lambda_2\rangle_{\mathit{l}}.$ Also, the dual of 2-D skew $(\lambda_1,\lambda_2)$-constacyclic codes 
is investigated.

Let $\mathbb{N}_0^2=\mathbb{N}_0\times\mathbb{N}_0$. Then $\mathbb{N}_0^2$ is a partial ordered set with $(i,j)\geq(k,l)$
if and only if $i\geq k$ and $j\geq l$. Moreover, $\mathbb{N}_0^2$ can be also totally ordered by a kind of lexicographic order ``$\Rightarrow$'',
where $(i,j)\Rightarrow(k,l)$ if and only if $j>l$ or both $j=l$ and $i\geq k$. Otherwise, $(i,j)\rightarrow(k,l)$ means $j>l$ or both $j=l$ and $i>k$.
Notice that $(i,j)\geq(k,l)$ implies that $(i,j)\Rightarrow(k,l)$, but the converse does not necessarily hold. 
A nonzero bivariate polynomial $f(x,y)\in R[x,y;\rho,\theta]$ is said
to have quasi-degree ${\rm deg}(f(x,y))=(k,t)$ if $f(x,y)$ has a nonzero term $a_{k,t}x^ky^t$ but does not have any nonzero term $a_{i,j}x^iy^j$
such that $(i,j)\rightarrow(k,t)$ holds. In this case, $a_{k,t}$ is called {\it the leading coefficient of} $f(x,y)$.
A bivariate skew polynomial is called {\it monic} provided its leading coefficient is $1$.
Let $\mathbb{F}_q$ be the Galois field with $q$ elements. For any polynomials $f(x,y)$ and $g(x,y)$ in $\mathbb{F}_q[x,y;\rho,\theta]$ we have
$${\rm deg}(f(x,y)\star g(x,y))={\rm deg}(f(x,y))+{\rm deg}(g(x,y)).$$
It is straightforward to see that $\mathbb{F}_q[x,y;\rho,\theta]$ has no nonzero zero-divisors.

\section{Basic properties of $R[x,y;\rho,\theta]$}
In this paper we denote by $R^{\rho,\theta}$ (resp. $R^{\rho\theta}$) the subring of $R$ that is fixed by $\rho,\theta$ (resp. $\rho\theta$).

Let $R$ be a commutative ring and $f(x)\in R[x]$. In \cite{Mc}, McCoy observed that if $0\neq g(x)\in R[x]$ be such that 
$f(x)g(x)=0$, then there exists a nonzero element $r$ of $R$ such that $f(x)r=0$.
Now, we state the bivariate skew version of the McCoy condition. 
\begin{thm}
Let $f(x,y)\in R^{\rho,\theta}[x,y;\rho,\theta]$. If $f(x,y)\star g(x,y)=0$ for some $0\neq g(x,y)\in R[x,y;\rho,\theta]$,
then there exists $0\neq r\in R$ such that $f(x,y)\star r=0.$
\end{thm}
\begin{proof}
Suppose that $f(x,y)=\sum\limits_{(i,j)\in \Lambda} f_{i,j}x^{i}y^{j}$. We can assume that
$g(x,y)$ is of the minimal degree ${\rm deg}(g(x,y))=(u,v)\big(\rightarrow(0,0)\big)$ (with respect to ``$\Rightarrow$''), and $g(x,y)$ has the leading coefficient $g_{u,v}$. By the minimality of ${\rm deg}(g(x,y))$ we have $f(x,y)\star g_{u,v}\neq0$. If $f_{i,j}g(x,y)=0$ for every $(i,j)\in\Lambda$,
then $f_{i,j}g_{u,v}=0$ for every $(i,j)\in\Lambda$.
Now, since $f(x,y)\in R^{\rho,\theta}[x,y;\rho,\theta]$, then $f(x,y)\star g_{u,v}=0$ 
which is a contradiction. Therefore $f_{i,j}g(x,y)\neq 0$ for some $(i,j)\in\Lambda$ and so assume that $(k,t)$ is the largest pair with this property (with respect to ``$\Rightarrow$''). Hence
$0=f(x,y)\star g(x,y)=(\sum\limits_{j=0}^{t}\sum\limits_{i=0}^{k} f_{i,j}x^{i}y^{j})\star g(x,y)$ which implies that 
$\rho^i\theta^j(f_{k,t}g_{u,v})=f_{k,t}\rho^i\theta^j(g_{u,v})=0$. Then $f_{k,t}g_{u,v}=0$.
Thus ${\rm deg}(g(x,y))\rightarrow{\rm deg}(f_{k,t}g(x,y))$. On the other hand $f(x,y)\star(f_{k,t}g(x,y))=f(x,y)\star(g(x,y)\star f_{k,t})=0$
which contradics the minimality of ${\rm deg}(g(x,y))$.
\end{proof}
The center of a ring $S$, denoted by $\mathrm{Z}(S)$, is the subset of $S$ consisting of all those elements in $S$ that commute with every element in $S$.
\begin{prop}\label{P1}
Let $g(x,y)\in\mathrm{Z}(R[x,y;\rho,\theta])$. 
Then $\langle x^ky^t\star g(x,y)\rangle_l$ is a two-sided ideal in $R[x,y;\rho,\theta]$ for every $k,t\in\mathbb{N}_{0}$.
\end{prop}
\begin{proof}
Let $g(x,y)\in \mathrm{Z}(R[x,y;\rho,\theta])$. Assume that $f(x,y)=\sum\limits_{j=0}^{n}\sum\limits_{i=0}^{m}a_{i,j}x^iy^j\\\in R[x,y;\rho,\theta]$. Then
\begin{eqnarray*}
\big(x^ky^t\star g(x,y)\big)\star f(x,y)&=&\big(x^ky^t\star g(x,y)\big)\star\sum\limits_{j=0}^{n}\sum\limits_{i=0}^{m}a_{i,j}x^iy^j\\
&=&\Big(x^ky^t\star\big(\sum\limits_{j=0}^{n}\sum\limits_{i=0}^{m}a_{i,j}x^iy^j\big)\Big)\star g(x,y)\\
&&(\mbox{since $g(x,y)$ is central})\\
&=&\big(\sum\limits_{j=0}^{n}\sum\limits_{i=0}^{m}\rho^k\theta^t(a_{i,j})x^{k+i}y^{t+j}\big)\star g(x,y)\\
&=&\Big(\big(\sum\limits_{j=0}^{n}\sum\limits_{i=0}^{m}\rho^k\theta^t(a_{i,j})x^{i}y^{j}\big)\star x^{k}y^{t}\Big)\star g(x,y)\\
&=&\big(\sum\limits_{j=0}^{n}\sum\limits_{i=0}^{m}\rho^k\theta^t(a_{i,j})x^{i}y^{j}\big)\star\big(x^{k}y^{t}\star g(x,y)\big),
\end{eqnarray*} 
which belongs to $\langle x^ky^tg(x,y)\rangle_{\mathit{l}}$. Hence, the result follows.
\end{proof}

\begin{prop}\label{P2}
Let $|\langle\rho\rangle|\mid l$ and $|\langle\theta\rangle|\mid s$. Then $R^{\rho,\theta}[x^l,y^s;\rho,\theta]\subseteq \mathrm{Z}(R[x,y;\rho,\theta])$.
\end{prop}
\begin{proof}
Let $a\in R^{\rho,\theta}$, $b\in R$ and $i,j,k,t\in\mathbb{N}_0$. 
Therefore 
\begin{eqnarray*} 
(ax^{li}y^{sj})\star(bx^ky^t)&=&a\rho^{li}\theta^{sj}(b)x^{li+k}y^{sj+t}\\
&=&bax^{k+li}y^{t+sj}\\
&=&(bx^ky^t)\star\big(\theta^{-t}\rho^{-k}(a)x^{li}y^{sj}\big)\\
&=&(bx^ky^t)\star(ax^{li}y^{sj})~~~(\mbox{since $a\in R^{\rho^{-1},\theta^{-1}}$}).
\end{eqnarray*} 
Consequently $R^{\rho,\theta}[x^l,y^s;\rho,\theta]\subseteq \mathrm{Z}(R[x,y;\rho,\theta])$.
\end{proof}


\begin{prop}\label{P3}
Let $\lambda$ be a unit in $R$. The following conditions hold:
\begin{enumerate}
\item  If $\langle x^l-\lambda\rangle_{\mathit{l}}$ is a two-sided ideal of $R[x,y;\rho,\theta]$, 
then $|\langle\rho\rangle|\mid l$ and $\rho(\lambda)=\lambda$.
\item If $\langle y^s-\lambda\rangle_{\mathit{l}}$ is a two-sided ideal of $R[x,y;\rho,\theta]$, then $|\langle\theta\rangle|\mid s$ and $\theta(\lambda)=\lambda$.
\end{enumerate}
\end{prop}
\begin{proof}
The proof is similar to that of \cite[Proposition 2.2]{JLU}.
\end{proof}

\begin{prop}
Let $\lambda_1,\lambda_2$ be units in $R$. The following conditions hold:
\begin{enumerate}
\item  If $|\langle\theta\rangle|\mid s$ and $\langle x^l-\lambda_1,y^s-\lambda_2\rangle_{\mathit{l}}$ is a two-sided ideal of $R[x,y;\rho,\theta]$, 
then $|\langle\rho\rangle|\mid l$ and $\rho(\lambda_1)=\lambda_1$.
\item If $|\langle\rho\rangle|\mid l$ and $\langle x^l-\lambda_1,y^s-\lambda_2\rangle_{\mathit{l}}$ is a 
two-sided ideal of $R[x,y;\rho,\theta]$, then $|\langle\theta\rangle|\mid s$ and $\theta(\lambda_2)=\lambda_2$.
\end{enumerate}
\end{prop}
\begin{proof}
Suppose that $|\langle\theta\rangle|\mid s$ and $\langle x^l-\lambda_1,y^s-\lambda_2\rangle_{\mathit{l}}$ is a two-sided ideal. Let $r\in R$. Then
$rx^l-r\lambda_1=r(x^l-\lambda_1)=(x^l-\lambda_1)r^{\prime}$ for some $r^{\prime}\in R$. The remainder is similar to the proof of \cite[Proposition 2.2]{JLU}.
\end{proof}

\begin{thm}\label{T1} 
Let $\lambda\in R$. The following conditions hold:
\begin{enumerate}
\item If $|\langle\rho\rangle|\mid l$ and $\lambda\in R^{\rho,\theta}$, 
then $x^l-\lambda$ is central in $R[x,y;\rho,\theta]$.
\item If $|\langle\theta\rangle|\mid s$ and $\lambda\in R^{\rho,\theta}$, 
then $y^s-\lambda$ is central in $R[x,y;\rho,\theta]$.
\item If $|\langle\rho\rangle|\mid l$, $|\langle\theta\rangle|\mid s$ and $\lambda_1,\lambda_2\in R^{\rho,\theta}$, 
then $\langle x^l-\lambda_1,y^s-\lambda_2\rangle_{\mathit{l}}$ is a two-sided ideal of $R[x,y;\rho,\theta]$.
\end{enumerate}
\end{thm}

\begin{proof}
$(1)$ Suppose that $|\langle\rho\rangle|\mid l$, $\rho(\lambda)=\lambda$ and $\theta(\lambda)=\lambda$. Let $a\in R$ and $k,t\in\mathbb{N}_0$.
Thus 
\begin{eqnarray*}
(ax^ky^t)\star(x^l-\lambda)&=&ax^{k+l}y^t- ax^ky^t\lambda\\
&=&ax^{l+k}y^t-\rho^{k}\theta^{t}(\lambda)ax^ky^t\\
&=&\big(x^{l}\star\rho^{-l}(a)x^{k}y^t\big)-\lambda ax^ky^t\\
&=&\big(x^{l}\star ax^{k}y^t\big)-\lambda ax^ky^t\\
&=&(x^l-\lambda)\star(ax^ky^t).
\end{eqnarray*}
Consequently $x^l-\lambda\in \mathrm{Z}(R[x,y;\rho,\theta])$.\\
(2) The proof if similar to that of (1).\\
(3) Is a direct consequence of parts (1) and (2).
\end{proof}

\begin{cor}\label{C1} 
Let $R$ be a ring. The following conditions hold:
\begin{enumerate}
\item Let $\theta(\lambda)=\lambda$. Then $|\langle\rho\rangle|\mid l$ and $\rho(\lambda)=\lambda$
if and only if $x^l-\lambda$ is central in $R[x,y;\rho,\theta]$.
\item Let $\rho(\lambda)=\lambda$. Then $|\langle\theta\rangle|\mid s$ and $\theta(\lambda)=\lambda$ 
if and only if $y^s-\lambda$ is central in $R[x,y;\rho,\theta]$.
\end{enumerate}
\end{cor}
\begin{proof}
By Proposition \ref{P3} and Theorem \ref{T1}.
\end{proof}

\begin{prop}\label{P4}
Let $f\star g$ be a monic central bivariat skew polynomial in $R[x,y;\rho,\theta]$ for some $f,g\in R[x,y;\rho,\theta]$.
Then $f\star g=g\star f$.
\end{prop}
\begin{proof}
Since $f\star g$ is central,
then we have $(g\star f)\star g=g\star (f\star g)=(f\star g)\star g.$ Therefore
$(g\star f-f\star g)\star g=0$. On the other hand $f\star g$ is monic, so
the leading coefficient of $g$ is unit. Hence $g$ is
not a zero-divisor. Consequently $f\star g=g\star f$ in $R[x,y;\rho,\theta]$.
\end{proof}

\begin{cor}
Let $|\langle\rho\rangle|\mid l$, $|\langle\theta\rangle|\mid s$ and $\lambda_1,\lambda_2\in R^{\rho,\theta}$.
If $(x^l-\lambda_1)\star(y^s-\lambda_2)=f\star g$ for some $f,g\in R[x,y;\rho,\theta]$, 
then $f\star g=g\star f$.
\end{cor}
\begin{proof}
By Theorem \ref{T1} and Proposition \ref{P4}.
\end{proof}


\section{2-D skew $(\lambda_1,\lambda_2)$-constacyclic codes}

\begin{defn}
Let $\mathcal{C}$ be a linear code over $R$ of length $ls$ whose codewords are viewed as $l\times s$ arrays,
i.e., $c\in \mathcal{C}$ is written as
$$c=\begin{pmatrix}
c_{0,0} & c_{0,1}&\dots & c_{0,s-1} \\
c_{1,0} & c_{1,1}&\dots & c_{1,s-1}\\
\vdots&\vdots&\ddots&\vdots\\
c_{l-1,0} &c_{l-1,1}&\dots & c_{l-1,s-1}
\end{pmatrix}.$$
We say that $\mathcal{C}$ is a {\it column skew $\lambda_1$-constacyclic code} 
if for every $l\times s$ array $c=\left(c_{i,j}\right)\in\mathcal{C}$
we have that
$$\begin{pmatrix}
\lambda_1\rho(c_{l-1,0})&\lambda_1\rho(c_{l-1,1})& \dots &\lambda_1\rho(c_{l-1,s-1})\\
\rho(c_{0,0})&\rho(c_{0,1})& \dots &\rho(c_{0,s-1}) \\
\vdots&\vdots&\ddots&\vdots\\
\rho(c_{l-2,0}) &\rho(c_{l-2,1}) & \dots &\rho(c_{l-2,s-1})
\end{pmatrix}\in \mathcal{C}.$$\\
Also, we say that $\mathcal{C}$ is a {\it row skew $\lambda_2$-constacyclic code} 
if for every $l\times s$ array $c=\left(c_{i,j}\right)\in\mathcal{C}$
we have that
$$\begin{pmatrix}
\lambda_2\theta(c_{0,s-1})&\theta(c_{0,0}) & \dots &\theta(c_{0,s-2}) \\
\lambda_2\theta(c_{1,s-1}) &\theta(c_{1,0}) & \dots &\theta(c_{1,s-2})\\
\vdots&\vdots&\ddots&\vdots\\
\lambda_2\theta(c_{l-1,s-1}) &\theta(c_{l-1,0}) & \dots &\theta(c_{l-1,s-2})
\end{pmatrix}\in \mathcal{C}.$$
If $\mathcal{C}$ is both column skew $\lambda_1$-constacyclic and row skew $\lambda_2$-constacyclic,
then we call $\mathcal{C}$ a {\it 2-D skew $(\lambda_1,\lambda_2)$-constacyclic code}.
\end{defn}

Define $R^{\circ}:=R[x,y;\rho,\theta]/\langle x^l-\lambda_1,y^s-\lambda_2\rangle_{\mathit{l}}.$
Consider the $R$-module isomorphism $R^{l\times s}\to R^{\circ}$ defined by $(a_{i,j})\mapsto\sum\limits_{j=0}^{s-1}\sum\limits_{i=0}^{l-1}a_{i,j}x^iy^j$
where $R^{l\times s}$ denotes the set of all $l\times s$ arrays. Then a codeword $c\in\mathcal{C}$ can be denoted by bivariate skew
polynomial $c(x,y)$ under above isomorphism.
\begin{thm}\label{T2}
A code $\mathcal{C}$ in $R^{\circ}$ is a 2-D skew $(\lambda_1,\lambda_2)$-constacyclic code if and only if $\mathcal{C}$ is a left $R[x,y;\rho,\theta]$-submodule of the left $R[x,y;\rho,\theta]$-module $R^{\circ}$.
\end{thm}
\begin{proof}
Similar to the proof of \cite[Theorem 3.2]{XH}.
\end{proof}

As a direct consequence of Theorem \ref{T1}(3) and Theorem \ref{T2} we have the next result. 
\begin{cor} 
Let $|\langle\rho\rangle|\mid l$, $|\langle\theta\rangle|\mid s$ and $\lambda_1,\lambda_2\in R^{\rho,\theta}$.  A code $\mathcal{C}$ in $R^{\circ}$ is a 2-D skew $(\lambda_1,\lambda_2)$-constacyclic code if and only if $\mathcal{C}$ is a left ideal of $R^{\circ}$.
\end{cor}

\begin{thm} 
Let $f_1(x,y),f_2(x,y)\in R[x,y;\rho,\theta]$ be two nonzero bivariate polynomials where $f_2(x,y)$ is monic. Provided that
${\rm deg}(f_1(x,y))\\\geq deg(f_2(x,y))$, there exists a pair of polynomials $h(x,y)(\neq0),g(x,y)\in R[x,y;\rho,\theta]$
which satisfy $f_1(x,y)=h(x,y)\star f_2(x,y)+g(x,y)$ such that $g(x,y)=0$ or ${\rm deg}(f_2(x,y))\nleq{\rm deg}(g(x,y))
\big(\leftarrow deg(f_1(x,y))\big)$. 
\end{thm}
\begin{proof}
The proof is similar to that of \cite[Theorem 2.1]{XH}.
\end{proof}

\begin{lem}
Let $\mathcal{C}$ be a 2-D skew $(\lambda_1,\lambda_2)$-constacyclic code in $R^{\circ}$ and $g(x,y)$
be a monic polynomial in
$R[x,y;\rho,\theta]$.
Then $g(x,y)$ is of the minimum degree (with respect to ``$\leq$'') in $\mathcal{C}$ if and only if $\mathcal{C}=\langle g(x,y)\rangle_{\mathit{l}}$.
\end{lem}
\begin{proof}
The ``if'' part is obvious. Let $g(x,y)$ be a monic polynomial of the minimum degree (with respect to ``$\leq$'')  in $\mathcal{C}$. If $c(x,y)\in\mathcal{C}$, then by the Division Algorithm in $R[x,y;\rho,\theta]$, there exists a pair of polynomials
$h(x,y)$$(\neq0)$, $r(x,y)\in R[x,y;\rho,\theta]$ which satisfy
$c(x,y)=h(x,y)\star g(x,y)+r(x,y)$, where either $r(x,y)=0$ or 
${\rm deg}(g(x,y))\nleq{\rm deg}(r(x,y))\big(\longleftarrow{\rm deg}(c(x,y))\big)$. As $\mathcal{C}$ is an $R[x,y;\rho,\theta]$-module,
$r(x,y)\in\mathcal{C}$ and the minimality of the degree of $g(x,y)$ implies $r(x,y)=0$.
\end{proof}

\begin{thm}
Let $\mathcal{C}$ be a 2-D skew $(\lambda_1,\lambda_2)$-constacyclic code in $\mathbb{F}_q^{\circ}$ and $g(x,y)$ be a 
monic generator polynomial of $\mathcal{C}$ in $\mathbb{F}_q[x,y;\rho,\theta]$. Then the following conditions hold:
\begin{enumerate}
\item  $g(x,y)$ divides $(x^l-\lambda_1)\star(y^s-\lambda_2)$.
\item Suppose that ${\rm deg}(g(x,y))=(l-k,s-t)$. Then
$$\mathbb{B}={\begin{Bmatrix}
g(x,y),&y\star g(x,y),&\dots&,y^{t-1}\star g(x,y) \\
x\star g(x,y),&xy\star g(x,y),&\dots&,xy^{t-1}\star g(x,y) \\
\vdots&\vdots&\ddots&\vdots\\
x^{k-1}\star g(x,y),&x^{k-1}y\star g(x,y),&\dots&,x^{k-1}y^{t-1}\star g(x,y) 
\end{Bmatrix}}$$
is a basis for $\mathcal{C}$ over $\mathbb{F}_q$ and so $|\mathcal{C}|=q^{kt}$.
\end{enumerate}
\end{thm}
\begin{proof}
(1) Assume that $g(x,y)$ does not divide $(x^l-\lambda_1)\star(y^s-\lambda_2)$.
By the Division Algorithm, there exist $f(x,y)(\neq0),r(x,y)\in \mathbb{F}_q[x,y;\rho,\theta]$ which satisfy
$$(x^l-\lambda_1)\star(y^s-\lambda_2)=f(x,y)\star g(x,y)+r(x,y)$$ such that either $r(x,y)=0$ or 
${\rm deg}(g(x,y))\nleq{\rm deg}(r(x,y))$. If ${\rm deg}(g(x,y))\\\nleq{\rm deg}(r(x,y))$, then
\begin{eqnarray*} 
r(x,y)&=&(x^l-\lambda_1)\star(y^s-\lambda_2)-f(x,y)\star g(x,y)\\
&\equiv&-f(x,y)\star g(x,y)~~\big(\mbox{mod}~\langle x^l-\lambda_1,y^s-\lambda_2\rangle_{\mathit{l}}\big)
\end{eqnarray*}
is also in $\mathcal{C}$ which contradicts the minimality of ${\rm deg}(g(x,y))$ with respect to ``$\leq$''.
Consequently $r(x,y)=0$ and so we are done.\\

(2) 
Let $0\neq f(x,y)\in \mathcal{C}=\langle g(x,y)\rangle_{\mathit{l}}$. Then there exists $q(x,y)\in \mathbb{F}_q[x,y;\rho,\theta]$ such that 
$f(x,y)=q(x,y)\star g(x,y)$. Notice that ${\rm deg}(f(x,y))\\\leq(l-1,s-1)$. Hence 
$${\rm deg}(q(x,y))={\rm deg}(f(x,y))-{\rm deg}(g(x,y))$$
$$\hspace{4.6cm}\leq(l-1,s-1)-(l-k,s-t)=(k-1,t-1).$$
Therefore, $\mathbb{B}$ is a generating set for $\mathcal{C}$.
Assume that $\{a_{i,j}|0\leq i\leq k-1, 0\leq j\leq t-1\}$ be a subset of $R$ such that 
$$\sum_{j=0}^{t-1}\sum_{i=0}^{k-1} a_{i,j}x^iy^j\star g(x,y)=0.$$
Given $r(x,y):=\sum\limits_{j=0}^{t-1}\sum\limits_{i=0}^{k-1} a_{i,j}x^iy^j$ such that $r(x,y)\neq0$, we have that ${\rm deg}(r(x,y))\leq(k-1,t-1)$ and $r(x,y)\star g(x,y)=0$. Therefore, there exist $q_1(x,y)$ and $q_2(x,y)$ in $\mathbb{F}_q[x,y;\rho,\theta]$ such that $$r(x,y)\star g(x,y)=q_1(x,y)\star(x^l-\lambda_1)+q_2(x,y)\star(y^s-\lambda_2).$$ Hence
{\scriptsize
\begin{eqnarray*}
{\rm deg}\big(q_1(x,y)\star(x^l-\lambda_1)+q_2(x,y)\star(y^s-\lambda_2)\big)&=&{\rm deg}(r(x,y))+{\rm deg}(g(x,y))\\
&\leq&(l-1,s-1),
\end{eqnarray*} }
which is a contradiction. Consequently $r(x,y)=0$
that implies $a_{i,j}=0$ for all $0\leq i\leq k-1$ and $0\leq j\leq t-1$.
\end{proof}

\begin{prop}
Let $R$ be a domain and $\lambda_1,\lambda_2$ be units in $R$. Suppose that $\mathcal{C}$ is a 2-D skew $(\lambda_1,\lambda_2)$-constacyclic 
code in $R^{\circ}$ that is generated by 
$$g(x,y)=\sum_{j=0}^{s-t-1}\sum_{i=0}^{l-k-1}g_{i,j}x^iy^j+\sum_{j=0}^{s-t}x^{l-k}y^{j}+\sum_{i=0}^{l-k-1}x^iy^{s-t}.$$ 
Then $g(x,y)\star xy\in\mathcal{C}$ if and only if $g(x,y)\in R^{\rho\theta}[x,y;\rho,\theta]$.
\end{prop}
\begin{proof}
Clearly $xy\star g(x,y)\in\mathcal{C}$. Assume that $g(x,y)\star xy\in\mathcal{C}$.
Therefore 
\begin{eqnarray*}
xy\star g(x,y)-g(x,y) xy&=&\sum_{j=0}^{s-t-1}\sum_{i=0}^{l-k-1}(\rho\theta(g_{i,j})-g_{i,j})x^{i+1}y^{j+1}\\&=&p(x,y)\star g(x,y),
\end{eqnarray*} 
for some $p(x,y)\in R[x,y;\rho,\theta]$. It is easy to see that $p(x,y)$ is constant and $p(x,y)g_{0,0}=0$.
Since $g(x,y)$ is a right divisor of $(x^l-\lambda_1)\star(y^s-\lambda_2)$ and $\lambda_1,\lambda_2$ are units, then 
$g_{0,0}$ is a unit. Hence $p(x,y)=0$ and so $\rho\theta(g_{i,j})=g_{i,j}$ for all $0\leq i\leq l-k-1$ and $0\leq j\leq s-t-1$.
Therefore $g(x,y)\in R^{\rho\theta}[x,y;\rho,\theta]$.\\
For the converse, assume that $g(x,y)\in R^{\rho\theta}[x,y;\rho,\theta]$. Thus $xy\star g_{i,j}=\rho\theta(g_{i,j})xy=g_{i,j}xy$
for all $0\leq i\leq l-k-1$ and $0\leq j\leq s-t-1$. Hence $g(x,y) xy=xy\star g(x,y)\in\mathcal{C}$, so we are done.
\end{proof}

\begin{defn}
For two $l\times s$ arrays $c=\left(c_{i,j}\right)$ and $d=(d_{i,j})$ we define:\\
$$c\odot d:=\sum\limits_{j=0}^{s-1}\sum\limits_{i=0}^{l-1} c_{i,j}d_{i,j}.$$
We say that $c$, $d$ are orthogonal if $c\odot d=0$.
The dual code of a linear code $\mathcal{C}$ of length $ls$ is the set of all $l\times s$ arrays orthogonal to all 
codewords of $\mathcal{C}$. The dual code of $\mathcal{C}$ is denoted by $\mathcal{C}^{\bot}$.
\end{defn}

\begin{thm}
Let $|\langle\rho\rangle|\mid l$, $|\langle\theta\rangle|\mid s$ and $\lambda_1,\lambda_2$ be units in $R$ such that $\lambda_1\in R^{\rho}, \lambda_2\in R^{\theta}$. Suppose that $\mathcal{C}$ is a code of lenght $ls$ over $R$. Then $\mathcal{C}$ is 2-D skew $(\lambda_1,\lambda_2)$-constacyclic if and only if $\mathcal{C}^{\bot}$ is 2-D skew $(\lambda_1^{-1},\lambda_2^{-1})$-constacyclic. In particular, if $\lambda_1^2=\lambda_2^2=1$,
then $\mathcal{C}$ is 2-D skew $(\lambda_1,\lambda_2)$-constacyclic if and only if
$\mathcal{C}^{\bot}$ is 2-D skew $(\lambda_1,\lambda_2)$-constacyclic.
\end{thm}

\begin{proof}
First of all, $\lambda_1\in R^{\rho}, \lambda_2\in R^{\theta}$ imply that $\lambda_1\in R^{\rho^{-1}}, \lambda_2\in R^{\theta^{-1}}$. Let $c=\left(c_{i,j}\right)\in \mathcal{C}$ and $b=(b_{i,j})\in \mathcal{C}^{\bot}$ be two $l\times s$ arrays.
Since $\mathcal{C}$ is a column $\lambda_1$-constacyclic code, then 
$$\begin{pmatrix}
\rho^{l-1}(\lambda_1 c_{1,0})&\rho^{l-1}(\lambda_1 c_{1,1})&\dots &\rho^{l-1}(\lambda_1 c_{1,s-1})\\
\rho^{l-1}(\lambda_1 c_{2,0})&\rho^{l-1}(\lambda_1 c_{2,1})&\dots &\rho^{l-1}(\lambda_1 c_{2,s-1})\\
\vdots&\vdots&\ddots&\vdots\\
\rho^{l-1}(\lambda_1 c_{l-1,0})&\rho^{l-1}(\lambda_1 c_{l-1,1})&\dots &\rho^{l-1}(\lambda_1 c_{l-1,s-1})\\
\rho^{l-1}(c_{0,0}) &\rho^{l-1}(c_{0,1}) & \dots &\rho^{l-1}(c_{0,s-1})
\end{pmatrix}$$ 
is in $\mathcal{C}$ and so it is orthogonal to $b=(b_{i,j})$, i.e.,
$$0=\sum_{i=1}^{l-1}\sum_{j=0}^{s-1}\rho^{l-1}(\lambda_1 c_{i,j})b_{i-1,j}+\sum_{j=0}^{s-1}\rho^{l-1}(c_{0,j})b_{l-1,j}$$
$$\hspace{1.5cm}=\lambda_1\left(\sum_{i=1}^{l-1}\sum_{j=0}^{s-1}\rho^{l-1}(c_{i,j})b_{i-1,j}+
\sum_{j=0}^{s-1}\rho^{l-1}(c_{0,j})\lambda_1^{-1}b_{l-1,j}\right)$$
Therefore 
$$\sum_{i=1}^{l-1}\sum_{j=0}^{s-1}\rho^{l-1}(c_{i,j})b_{i-1,j}+
\sum_{j=0}^{s-1}\rho^{l-1}(c_{0,j})\lambda_1^{-1}b_{l-1,j}=0.$$
Hence 
$$0=\rho(0)=\sum_{i=1}^{l-1}\sum_{j=0}^{s-1}c_{i,j}\rho(b_{i-1,j})+
\sum_{j=0}^{s-1}c_{0,j}\lambda_1^{-1}\rho(b_{l-1,j}),$$
which shows that 
$$\begin{pmatrix}
\lambda_1^{-1}\rho(b_{l-1,0})&\lambda_1^{-1}\rho(b_{l-1,1})& \dots &\lambda_1^{-1}\rho(b_{l-1,s-1})\\
\rho(b_{0,0})&\rho(b_{0,1})& \dots &\rho(b_{0,s-1}) \\
\vdots&\vdots&\ddots&\vdots\\
\rho(b_{l-2,0}) &\rho(b_{l-2,1}) & \dots &\rho(b_{l-2,s-1})
\end{pmatrix}$$ 
is orthogonal to $c=(c_{i,j})$ and so it belongs to
$\mathcal{C}^{\bot}$. Thus $\mathcal{C}^{\bot}$ is column $\lambda_1^{-1}$-constacyclic.
On the other hand, since $\mathcal{C}$ is a row $\lambda_2$-constacyclic code, then 
$$\begin{pmatrix}
\theta^{s-1}(\lambda_2 c_{0,1})&\dots&\theta^{s-1}(\lambda_2 c_{0,s-1})&\theta^{s-1}(c_{0,0})\\
\theta^{s-1}(\lambda_2 c_{1,1})&\dots&\theta^{s-1}(\lambda_2 c_{1,s-1})&\theta^{s-1}(c_{1,0})\\
\vdots&\ddots&\vdots&\vdots\\
\theta^{s-1}(\lambda_2 c_{l-1,1})&\dots&\theta^{s-1}(\lambda_2 c_{l-1,s-1})&\theta^{s-1}(c_{l-1,0})
\end{pmatrix}$$ 
is in $\mathcal{C}$ and so it is orthogonal to $b=(b_{i,j})$, i.e.,
$$0=\sum_{j=1}^{s-1}\sum_{i=0}^{l-1}\theta^{s-1}(\lambda_2 c_{i,j})b_{i,j-1}+\sum_{i=0}^{l-1}\theta^{s-1}(c_{i,0})b_{i,s-1}$$
$$\hspace{1.5cm}=\lambda_2\left(\sum_{j=1}^{s-1}\sum_{i=0}^{l-1}\theta^{s-1}(c_{i,j})b_{i,j-1}+\sum_{i=0}^{l-1}\theta^{s-1}(c_{i,0})\lambda_2^{-1}b_{i,s-1}\right)$$
Thus
$$\sum_{j=1}^{s-1}\sum_{i=0}^{l-1}\theta^{s-1}(c_{i,j})b_{i,j-1}+\sum_{i=0}^{l-1}\theta^{s-1}(c_{i,0})\lambda_2^{-1}b_{i,s-1}=0.$$
So we have  
$$0=\theta(0)=\sum_{j=1}^{s-1}\sum_{i=0}^{l-1}c_{i,j}\theta(b_{i,j-1})+\sum_{i=0}^{l-1}c_{i,0}
\lambda_2^{-1}\theta(b_{i,s-1}),$$
which shows that 
$$\begin{pmatrix}
\lambda_2^{-1}\theta(b_{0,s-1})&\theta(b_{0,0}) & \dots &\theta(b_{0,s-2}) \\
\lambda_2^{-1}\theta(b_{1,s-1}) &\theta(b_{1,0}) & \dots &\theta(b_{1,s-2})\\
\vdots&\vdots&\ddots&\vdots\\
\lambda_2^{-1}\theta(b_{l-1,s-1}) &\theta(b_{l-1,0}) & \dots &\theta(b_{l-1,s-2})
\end{pmatrix}$$
is orthogonal to $c=(c_{i,j})$ and so it belongs to
$\mathcal{C}^{\bot}$. Thus, it follows that $\mathcal{C}^{\bot}$ is row $\lambda_2^{-1}$-constacyclic.
Consequently $\mathcal{C}^{\bot}$ is 2-$D$ skew $(\lambda_1^{-1},\lambda_2^{-1})$-constacyclic.\\
The converse holds by the fact that $(\mathcal{C}^{\bot})^{\bot}=\mathcal{C}$.
\end{proof}

The ring $R[x,y;\rho,\theta]$ can be localized to the right at the multiplicative set $S=\{x^iy^j\mid i,j\in\mathbb{N}\}$. The existance of the localization 
$R[x,y;\rho,\theta]S^{-1}$ follows from \cite[Theorem 2]{R} since $S$ verifies the following two necessary and sufficient conditions:
\begin{enumerate}
\item For all $x^{i_1}y^{j_1}\in S$ and $f(x,y)\in R[x,y;\rho,\theta]$, there exists $x^{i_2}y^{j_2}\in S$ and $g(x,y)\in R[x,y;\rho,\theta]$ such that 
$f(x,y)x^{i_1}y^{j_1}= x^{i_2}y^{j_2}\star g(x,y)$. To prove this note to the multiplication rule $x^iy^j\star a=\rho^i\theta^j(a)x^{i}y^{j}$.
\item If for $x^{i_1}y^{j_1}\in S$ and $f(x,y)\in R[x,y;\rho,\theta]$ we have $x^{i_1}y^{j_1}\star f(x,y)=0$, then there exists $x^{i_2}y^{j_2}\in S$  such that 
$f(x,y)x^{i_2}y^{j_2}=0$. But since $x^{i_2}y^{j_2}$ is never a zero divisor, $f(x,y)$ must be zero.
\end{enumerate}
Now, we consider the ring $R[x,y;\rho,\theta]S^{-1}$ consisting of the elements $\sum\limits_{j=0}^{t}\sum\limits_{i=0}^{k}x^{-i}y^{-j}a_{i,j}$, 
where the coefficients are on the right and where the multiplication rule is given by $ax^{-1}y^{-1}=x^{-1}y^{-1}\rho\theta(a)$.

\begin{prop}
Let $\psi:R[x,y;\rho,\theta]\to R[x,y;\rho,\theta]S^{-1}$ be defined by 
$$\psi(\sum\limits_{j=0}^{t}\sum\limits_{i=0}^{k}a_{i,j}x^iy^j)=\sum\limits_{j=0}^{t}\sum\limits_{i=0}^{k}x^{-i}y^{-j}a_{i,j}.$$
Then $\psi$ is a ring anti-isomorphism.
\end{prop}
\begin{proof}
Similar to \cite[Theorem 4.4]{BSU}.
\end{proof}

{\bf From now on, we assume that $ \lambda_1,\lambda_2\in R^{\rho,\theta}$, $\mid\langle\rho\rangle\mid|l$ and $\mid\langle\theta\rangle\mid|s$.}

\begin{thm}\label{main}
Let $a(x,y)=\sum\limits_{j=0}^{s-1}\sum\limits_{i=0}^{l-1}a_{i,j}x^iy^j$ and $b(x,y)=
\sum\limits_{j=0}^{s-1}\sum\limits_{i=0}^{l-1}b_{i,j}x^iy^j$ be in $R[x,y;\rho,\theta]$.
Suppose that $\lambda_1^2=\lambda_2^2=1$. Then the following conditions are equivalent:
\begin{enumerate}
\item The coefficient matrix of $a(x,y)$ is orthogonal to the coefficient matrix of
$x^iy^j\big(x^{l-1}y^{s-1}\star\psi(b(x,y))\big)$ for all $i\in\{0,1,\dots,l-1\}$ and all $j\in\{0,1,\dots,s-1\}$;
\item The coefficient matrix of $a(x,y)$ is orthogonal to 
$$\mathcal{A}=
{\begin{pmatrix}
b_{l-1,s-1} &\theta(b_{l-1,s-2})&\dots&\theta^{s-1}(b_{l-1,0})\\
\rho(b_{l-2,s-1})&\rho\theta(b_{l-2,s-2})&\dots&\rho\theta^{s-1}(b_{l-2,0})\\
\vdots&\vdots&\ddots&\vdots\\
\rho^{l-1}(b_{0,s-1})&\rho^{l-1}\theta(b_{0,s-2})&\dots&\rho^{l-1}\theta^{s-1}(b_{0,0})
\end{pmatrix}}$$
and all of its column skew $\lambda_1$-constacyclic shifts and row skew $\lambda_2$-constacyclic shifts;
\item $a(x,y)\star b(x,y)=0$ in $R^{\circ}$.
\end{enumerate} 
\end{thm}
\begin{proof}
(1)$\Leftrightarrow$(2) With a routine computation we can deduce that the coefficient matrix of $x^{l-1}y^{s-1}\star\psi(b(x,y))$ is equal to $\mathcal{A}$,
also the coefficient matrices of  $x^iy^j\big(x^{l-1}y^{s-1}\star\psi(b(x,y))\big)$ for $0\leq i\leq l-1$ and $0\leq j\leq s-1$ are precisely 
column skew $\lambda_1$-constacyclic shifts and row skew $\lambda_2$-constacyclic shifts of $\mathcal{A}$.\\
(2)$\Leftrightarrow$(3)
Suppose that $a(x,y)\star b(x,y)=\sum\limits_{t=0}^{s-1}\sum\limits_{k=0}^{l-1}c_{i,j}x^iy^j\in R^{\circ}$. For $k\in\{0,1,\dots,l-1\}$ and $t\in\{0,1,\dots,s-1\}$
we have that
{\scriptsize
\begin{eqnarray*}
c_{k,t}&=&\sum_{\substack{j+v=t\\0\leq j,v\leq s-1}}\ \sum_{\substack{i+u=k\\0\leq i,u\leq l-1}}a_{i,j}\rho^i\theta^j(b_{uv})+\sum_{\substack{j+v=t\\0\leq j,v\leq s-1}}\ \sum_{\substack{i+u=k+l\\0\leq i,u\leq l-1}}a_{i,j}\rho^i\theta^j(\lambda_1 b_{uv})\\
&+&\sum_{\substack{j+v=t+s\\0\leq j,v\leq s-1}}\ \sum_{\substack{i+u=k\\0\leq i,u\leq l-1}}a_{i,j}\rho^i\theta^j(\lambda_2 b_{uv})+\sum_{\substack{j+v=t+s\\0\leq j,v\leq s-1}}\ \sum_{\substack{i+u=k+l\\0\leq i,u\leq l-1}}a_{i,j}\rho^i\theta^j(\lambda_1\lambda_2 b_{uv})\\
&=&\lambda_1\lambda_2\big(\sum_{\substack{j+v=t\\0\leq j,v\leq s-1}}\ \sum_{\substack{i+u=k\\0\leq i,u\leq l-1}}a_{i,j}\rho^{k-u}\theta^{t-v}(\lambda_1\lambda_2b_{uv})\\
&+&\sum_{\substack{j+v=t\\0\leq j,v\leq s-1}}\ \sum_{\substack{i+u=k+l\\0\leq i,u\leq l-1}}a_{i,j}\rho^{k+l-u}\theta^{t-v}(\lambda_2 b_{uv})\\
&+&\sum_{\substack{j+v=t+s\\0\leq j,v\leq s-1}}\ \sum_{\substack{i+u=k\\0\leq i,u\leq l-1}}a_{i,j}\rho^{k-u}\theta^{t+s-v}(\lambda_1 b_{uv})\\
&+&\sum_{\substack{j+v=t+s\\0\leq j,v\leq s-1}}\ \sum_{\substack{i+u=k+l\\0\leq i,u\leq l-1}}a_{i,j}\rho^{k+l-u}\theta^{t+s-v}(b_{uv})\big)
\end{eqnarray*}}
which is the $\odot$-product of 
$$\scriptsize{\lambda_1\lambda_2\begin{pmatrix}
a_{0,0} & \dots & a_{0,s-1} \\
a_{1,0} & \dots & a_{1,s-1}\\
\vdots&\ddots&\vdots\\
a_{l-1,0} & \dots & a_{l-1,s-1}
\end{pmatrix}}$$
in the matrix

$$
\scriptsize
{\begin{pmatrix}
\lambda_1\lambda_2b_{k,t} &\theta(\lambda_1\lambda_2b_{k,t-1})\dots&\theta^t(\lambda_1\lambda_2b_{k,0})&\theta^{t+1}(\lambda_1b_{k,s-1})\dots&\theta^{s-1}(\lambda_1b_{k,t+1})\\
\rho(\lambda_1\lambda_2b_{k-1,t})&\rho\theta(\lambda_1\lambda_2b_{k-1,t-1})\dots&\rho\theta^t(\lambda_1\lambda_2b_{k-1,0})&\rho\theta^{t+1}(\lambda_1b_{k-1,s-1})\dots&\rho\theta^{s-1}(\lambda_1b_{k-1,t+1})\\
\vdots&\vdots&\vdots&\vdots&\vdots\\
\rho^k(\lambda_1\lambda_2b_{0,t})&\rho^k\theta(\lambda_1\lambda_2b_{0,t-1})\dots&\rho^k\theta^t(\lambda_1\lambda_2b_{0,0})&\rho^k\theta^{t+1}(\lambda_1b_{0,s-1})\dots&\rho^k\theta^{s-1}(\lambda_1b_{0,t+1})\\
\rho^{k+1}(\lambda_2b_{l-1,t})&\rho^{k+1}\theta(\lambda_2b_{l-1,t-1})\dots&\rho^{k+1}\theta^t(\lambda_2b_{l-1,0})&\rho^{k+1}\theta^{t+1}(b_{l-1,s-1})\dots&\rho^{k+1}\theta^{s-1}(b_{l-1,t+1})\\
\vdots&\vdots&\vdots&\vdots&\vdots\\
\rho^{l-1}(\lambda_2b_{k+1,t})&\rho^{l-1}\theta(\lambda_2b_{k+1,t-1})\dots&\rho^{l-1}\theta^t(\lambda_2b_{k+1,0})&\rho^{l-1}\theta^{t+1}(b_{k+1,s-1})\dots&\rho^{l-1}\theta^{s-1}(b_{k+1,t+1})
\end{pmatrix}}.$$\\
Therefore $a(x,y)\star b(x,y)=0$ in $R^{\circ}$ if and only if $c_{k,t}=0$
for all $k\in\{0,1,\dots,l-1\}$ and all $t\in\{0,1,\dots,s-1\}$ which is true if and only if 
$${\begin{pmatrix}
a_{0,0} & \dots & a_{0,s-1} \\
a_{1,0} & \dots & a_{1,s-1}\\
\vdots&\ddots&\vdots\\
a_{l-1,0} & \dots & a_{l-1,s-1}
\end{pmatrix}}$$
is orthogonal to 
$\mathcal{A}$
and all of its column skew $\lambda_1$-constacyclic shifts and row skew $\lambda_2$-constacyclic shifts. 
\end{proof}

\begin{prop}\label{haveopen}
Let $\lambda_1^2=\lambda_2^2=1$, $g(x,y)$ be a right divisor of $(x^l-\lambda_1)\star(y^s-\lambda_2)$ in $R[x,y;\rho,\theta]$ and $h(x,y):=\frac{(x^l-\lambda_1)\star(y^s-\lambda_2)}{g(x,y)}$ with ${\rm deg}(h(x,y))=(k,t)$. Suppose that $\mathcal{C}$ is a 2-D skew $(\lambda_1,\lambda_2)$-constacyclic code of length $ls$ over $R$ that is generated by $g(x,y)$.
Then the bivariate skew polynomial $x^ky^t\star\psi(h(x,y))$ is a right divisor of $(x^l-\lambda_1)\star(y^s-\lambda_2)$ and $x^ky^t\star\psi(h(x,y))\in \mathcal{C}^{\bot}$.
\end{prop}
\begin{proof}
By the assumtions that $\lambda_1,\lambda_2\in R^{\rho,\theta}$, $\mid\langle\rho\rangle\mid|l$ and $\mid\langle\theta\rangle\mid|s$ we can see that
{\scriptsize
\begin{eqnarray*} 
\Big(\psi(g(x,y))\star\lambda_1\lambda_2x^{l-k}y^{s-t}\Big)\star\Big(x^ky^t\star\psi(h(x,y))\Big)
&=&\psi(g(x,y))\star\lambda_1\lambda_2x^{l}y^{s}\star\psi(h(x,y)) \\ 
&=& \lambda_1\lambda_2 x^{l}y^{s} \star\psi(g(x,y))\star\psi(h(x,y)) \\ 
&=& \lambda_1\lambda_2 x^{l}y^{s} \star\psi(h(x,y)\star g(x,y))\\
&&\hspace{-0.7cm}(\mbox{since $\psi$ is a ring anti-isomorphism})\\
&=& \lambda_1\lambda_2 x^{l}y^{s} \star\psi\Big((x^l-\lambda_1)\star(y^s-\lambda_2)\Big)\\
&=& \lambda_1\lambda_2 x^{l}y^{s}\star(x^{-l}y^{-s}\\
&-&x^{-l}\lambda_2-y^{-s}\lambda_1+\lambda_1\lambda_2)\\
&=&(x^l-\lambda_1)\star(y^s-\lambda_2).
\end{eqnarray*}}
Moreover, since $g(x,y)\star h(x,y)=0$ in $R^{\circ}$, then Theorem \ref{main} implies that $x^ky^t\star\psi(h(x,y))\in \mathcal{C}^{\bot}$.
\end{proof}

Now we can state the following open problem.
\begin{prob}
In Proposition \ref{haveopen}, is $\mathcal{C}^{\bot}$ generated by $x^ky^t\star\psi(h(x,y))$?
\end{prob}

\begin{thm}
Let $g(x,y)$ be a monic right divisor of $(x^l-\lambda_1)\star(y^s-\lambda_2)$ in $R[x,y;\rho,\theta]$ and $h(x,y):=\frac{(x^l-\lambda_1)\star(y^s-\lambda_2)}{g(x,y)}$. Suppose that $\mathcal{C}$ is a 2-D skew $(\lambda_1,\lambda_2)$-constacyclic code of length $ls$ over $R$ that is generated by $g(x,y)$.
Then for $f(x,y)\in R[x,y;\rho,\theta]$, $f(x,y)\in \mathcal{C}$ if and only if $f(x,y)\star h(x,y)=0$ in $R^{\diamond}=R[x,y;\rho,\theta]/\langle (x^l-\lambda_1)\star(y^s-\lambda_2)\rangle_{\mathit{l}}$.
\end{thm}
\begin{proof}
Let $f(x,y)\in\mathcal{C}$. Then $f(x,y)=q(x,y)\star g(x,y)$ for some $q(x,y)\in R[x,y;\rho,\theta]$. So we have
\begin{eqnarray*} 
f(x,y)\star h(x,y)&=&q(x,y)\star g(x,y)\star h(x,y)\\
&=&q(x,y)\star(x^l-\lambda_1)\star(y^s-\lambda_2)=0,
\end{eqnarray*} 
in $R^{\diamond}$.
Conversely, assume that $f(x,y)\star h(x,y)=0$ in $R^{\diamond}$.
Then $f(x,y)\star h(x,y)=q(x,y)\star(x^l-\lambda_1)\star(y^s-\lambda_2)$ for some $q(x,y)\in R[x,y;\rho,\theta]$.
Therefore $f(x,y)\star h(x,y)=q(x,y)\star g(x,y)\star h(x,y)$. Now, since $h(x,y)$ is monic, $f(x,y)=q(x,y)\star g(x,y)\in \mathcal{C}$. 
\end{proof}


\vskip 0.4 true cm

\begin{center}{\textbf{Acknowledgments}}
\end{center}
The author would like to thank the referee for constructive comments 
which will help to improve the quality of the paper.  \\ \\
\vskip 0.4 true cm

\bibliographystyle{amsplain}

\bigskip
\bigskip

{\footnotesize {\bf Hojjat Mostafanasab}\; \\ {Eski Silahtara\v{g}a Elektrik Santrali, Kazim Karabekir}, 
{Istanbul Bilgi University}, {Cad. No: 2/1334060, Ey\"{u}p Istanbul, Turkey.}\\
{\tt Email: h.mostafanasab@gmail.com}\\

\end{document}